\documentclass[a4paper,11pt]{amsart}
\usepackage[english]{babel}
\usepackage{amsmath,amssymb,amsthm,mathrsfs,amsopn,amscd}

\usepackage{verbatim,graphicx,color}
\usepackage[utf8]{inputenc}
\usepackage{indentfirst}
\usepackage[T1]{fontenc}
\usepackage{fancyhdr}

\newtheorem{thm}{Theorem}[section]
\newtheorem{lem}[thm]{Lemma}
\newtheorem{pro}[thm]{Proposition}
\newtheorem{cor}[thm]{Corollary}

\def\ep{\varepsilon}
\def\SSN{{\mathbb S}^{N-1}}

\def\RN{\mathbb{R}^N}

\def\la{\lambda}

\def\ka{\kappa}
\def\pa{\partial}

\def\de{\delta}

\newcommand{\La}{\Lambda}

\newcommand{\om}{\omega}
\newcommand{\acc}{\`}

\newcommand{\RE}{\mathbb R}

\newcommand{\ovr}{\overline}

\title[Characterization of ellipsoids]{Characterization of ellipsoids \\
as $K$-dense sets}

\author[Magnanini]{Rolando Magnanini}
\address{Dipartimento di Matematica e Informatica ``U. Dini'', Universit\acc a di Firenze, viale Morgagni 67/A, 50134 Firenze, Italy}
\email{magnanin@math.unifi.it}

\author[Marini]{Michele Marini}
\address{Scuola Normale Superiore, Piazza dei Cavalieri 7, 56126 Pisa, Italy}
\email{michele.marini@sns.it}


\begin{document}

\begin{abstract}
Let $K\subset\RN$ be any convex body containing the origin.
A measurable set $G\subset\RN$ with finite and positive Lebesgue measure is
said to be $K$-dense if, for
any fixed $r>0,$ the measure of $G\cap (x+r K)$ is constant when $x$ varies on the boundary of $G$ 
(here, $x+r K$ denotes a translation of a dilation of $K$). 
In \cite{MM}, we proved for the case in which $N=2$ that if $G$ is $K$-dense, then both $G$ and $K$ 
must be homothetic to the same ellipse. 
Here,  we completely characterize $K$-dense sets in $\RN$: if $G$ is $K$-dense, then both $G$ and $K$ 
must be homothetic to the same ellipsoid. Our proof, by building upon results obtained in \cite{MM},  relies on an asymptotic formula for the measure of $G\cap (x+r K)$ for large values of the parameter $r$ 
and a classical characterization of ellipsoids due to C.M. Petty \cite{Pe}.
\end{abstract}

\maketitle

\section{Introduction}

Let $K$ be a convex body containing the origin of $\RN$ and $G$ be a measurable subset of $\RN$  
with finite positive Lebesgue measure $V(G).$ 
We say that $G$ is {\it $K$-dense} if 
there is a function $c:(0,\infty)\to(0,\infty)$ such that 
\begin{equation}
\label{Kdense}
V(G\cap (x+r\,K))=c(r) \ \mbox{ for } \ x\in\pa G, \ r>0.
\end{equation}
Here, $\pa G$ is the topological boundary of $G$ and $x+r\,K$ denotes the translation by a vector $x$ of
a dilation of $K$ by a factor $r>0$. When $K$ is the unit ball, $K$-dense sets were studied
in \cite{MPS} in connection with the so-called stationary isothermic (or time-invariant level) surfaces of solutions of the heat equation (see also \cite{MS} for a related paper).
\par
Plane $K$-dense sets have been characterized in \cite{ABG} and \cite{MM}. They cannot exist unless 
they are homothetic to $K$ itself and, if this is the case, they must be ellipses (together with $K$).  
In this paper, we shall extend that characterization to general dimension by proving
the following result.
\begin{thm}
\label{th:main}
Let $K\subset\RN$ be a convex body and assume that there is a set $G\subset\RN$ of finite
positive measure such that \eqref{Kdense} holds.
\par
Then, both $K$ and $G$ must be homothetic to the same ellipsoid. 
\end{thm}
\par
The case $N=2$ was first settled in \cite{ABG} under some smoothness assumptions ($\pa K$ of class $C^2$ and $\pa G$ of class $C^4$). It should also be noticed that the proof in \cite{ABG} works even if condition \eqref{Kdense} holds when $r$ ranges in a sufficiently small interval $(0, r_0)$, since 
it only uses local information on $\pa G$. 
\par
In \cite{MM}, we were able to remove such regularity assumptions. In fact, we showed that in the plane
the occurrence of property \eqref{Kdense} implies that both $\pa G$ and $\pa K$ are
necessarily of class $C^\infty$. Moreover, we gave an alternative proof of the characterization
which is based on some local information on $\pa G$ derived from \eqref{Kdense} and classical affine inequalities for convex bodies.
\par
In \cite{MM}, we also established some facts that hold in general dimension and will be useful in the 
remainder of this paper: let $K\subset\RN$ be a convex body and assume that property \eqref{Kdense} holds, then 
\begin{itemize}
\item[(i)] $G$ is strictly convex;
\item[(ii)] $\pa G$ is at least of class $C^{1,1}$;
\item[(iii)] if $K$ is centrally symmetric (i.e. $-K=K$), then $K=G-G$ up to dilations, $K$
is strictly convex and $\pa K$ is at least of class $C^{1,1}$;
\item[(iv)] if $\pa G$ is differentiable at $x$, then
$$
V(G\cap (x+r\,K))=V_0(x)\,r^N+o(r^{N}) \ \mbox{ as } \ r\to 0^+;
$$
\item[(v)] if $\pa G$ is of class $C^2$ in a neighborhood of $x$, then
$$
V(G\cap (x+r\,K))=V_0(x)\,r^N+V_1(x)\,r^{N+1}+o(r^{N+1}) \ \mbox{ as } \ r\to 0^+.
$$
\end{itemize}
The coefficients $V_0(x)$ and $V_1(x)$ are explicitly computed; $G-G$ denotes 
the {\it Minkowski sum} of $G$ and $-G$: $G-G=G+(-G)=\{x-y: x, y\in G\}$.
\par
It will be useful to understand the mechanism of our proof in \cite{MM}. Since \eqref{Kdense} holds, (ii) and (iv)
imply that the function $V_0$ is constant on $\pa G$. By the explicit expression of $V_0(x)$ then one gets that
\begin{equation}
\label{cond1}
V(\{ y\in K: y\cdot\nu(x)\ge 0\})=\frac12\,V(K) \ \mbox{ for every } \ x\in\pa G,
\end{equation}
where $\nu(x)$ denotes the exterior unit normal to $\pa G$ at $x$.
When $N=2$, thanks to (i), it is not difficult to show that \eqref{cond1} implies
that $K$ is centrally symmetric --- indeed, that is also true for $N\ge 3$, by a non-trivial
result of Schneider \cite{Sc}.
Thus, (iii) comes into play and we can infer 
further regularity
($C^{2,1}$) for $\pa G$. Hence, (v) can be used: also the function $V_1$ must be constant
on $\pa G$. This condition gives a pointwise constraint on the curvature of $\pa G$  (see \cite[(1.8)]{MM}) that --- for $N=2$ ---
ensures that $K=2G$ up to homotheties and, with the help
of Minkowski's inequality for mixed volumes and an inequality involving the {\it affine surface area} of $\pa G$, gives the desired conclusion.
\par
Now, let us look at the case in which $N\ge 3$. Of course, (i) and (ii) still hold, if $G$ is $K$-dense.
Thus, the formula in (iv) still makes sense and hence, by the aforementioned result \cite{Sc}, $K$
is centrally symmetric; consequently, (iii) holds, too. 
Therefore, also (v) makes sense and, even now, we can deduce that $V_1$ must be constant on $\pa G$.
Unfortunately, the pointwise constraint on the principal curvatures \cite[(1.8)]{MM} is no longer enough 
to deduce that $K=2G$ and to conclude.
\par
In this paper, we succeed in our purpose by changing strategy: we give up the asymptotic expansion
for $r\to 0^+$ in (v) in favour of an expansion like
\begin{equation}
\label{r-large}
V(G\cap (x+r\,K))=V(G)+W(x)\,(r_G-r)^{\frac{N+1}{2}}+o\left((r_G-r)^{\frac{N+1}{2}}\right)
\end{equation}
 as $r\to r_G^-$, where
$$
r_G=\inf\{r>0: G\subseteq x+r\,K\}, \ x\in \pa G.
$$
Notice that, if $G$ is $K$-dense, then $r_G$ is independent on $x\in\pa G$; since our problem is invariant with respect to dilations of $K$, throughout the paper, we shall assume that $r_G=1$.
\par
The computation of the coefficient $W(x)$ is carried out in Section 2 and involves 
the {\it support function} $h_K:\SSN\to\RE$
of the convex body $K$ with respect to the origin and the {\it shape operators}
$S_G$ and $S_K$ are of $G$ and $K$, respectively. In fact, iIt turns out that
for $x\in\pa G$
\begin{equation}
\label{W1}
W(x)=-\frac{2\omega_{N-1}\, h_K(u)^{\frac{N+1}{2}}}{{(N^2-1)}\,\det[S_G(u)-S_K(u)]^{\frac 1 2}} \ \mbox{ with } \ u=\nu(\ovr x);
\end{equation}
here, $\{\ovr x\}=\pa G\cap (x+K)$,\footnote{It will be made clear in Section 2 that $\ovr x$ is uniquely determined.} $\nu(\ovr x)$ is the exterior unit normal to $\pa G$ at $\ovr x$, 
 and $\om_{N-1}$ denotes the surface measure of the unit sphere $\mathbb S^{N-2}$ of $\mathbb R^{N-1}$.
\par
Properties (i) and \eqref{Kdense} imply that the right-hand side of \eqref{W1} must be constant
as a function of $u\in\SSN$. A first consequence of this fact is that $K=2G$ up to homotheties;
a second consequence is that
\begin{equation}
\label{Petty}
\ka_G(u)=c\,h_G(u)^{N+1} \ \mbox{ for every } \ u\in\SSN,
\end{equation}
for some positive constant $c$; here, $\ka_G$ denotes the {\it Gauss curvature} of $\pa G$
at the (unique) point $x\in\pa G$ having normal equal to $u$.
\par
The identity \eqref{Petty} is well-known in the theory of convex bodies: in fact, C.M. Petty
proved in \cite{Pe} that it characterizes 
$G$ as an ellipsoid. 
\par
Section 2 contains all the details.

\setcounter{equation}{0}

\section{The proof of Theorem \ref{th:main}}

Let $G\subset\RN$ be a $C^2$ convex body. In a sufficiently small neighborhood of a point $x\in\pa G$, the set $\pa G$ is the graph of a $C^2$-regular convex function over the tangent space to $\pa G$ at $x$; we denote by $S_G$ the Hessian of this function (the bilinear form associated is often called \textit{shape operator}); it is well-known that its determinant $\ka_G$ is the {\it Gaussian curvature} of $\pa G$ at that point.
When $G$ is strictly convex, without any ambiguity we can think of $S_G$ as a function over the unit sphere, so that, for a given $u\in\mathbb S^{N-1}$, $S_G(u)$ denotes the shape operator at the only point $x\in\pa G$ with outward unit normal equal to $u$.
\par We know from \cite{MM} that, if $G$ is $K$-dense, then $\pa G$ is of class $C^2$ but, unfortunately, we can not assert that $\pa K$ is of class $C^2$, even if we know that $K=G-G$ (see \cite{B}, for instance).
Nevertheless, in \cite{KP} it is shown that, if $G$ is strongly convex\footnote{That is $S_K(u)>0$ 
for every $u\in\SSN$ ---
with which we mean that $S_K(u)$ is positive definite for every $u\in\SSN$.}, then $K$ has the same regularity as $G$; in particular the
following result holds.

\begin{thm}[S. Krantz, H. Parks]\label{KrPa}
If $A$ is a strongly convex body with boundary of class $C^\infty$ and $B$ is a convex body
with boundary of class $C^2$, then
the Minkowski sum $A+B$ has boundary of class $C^\infty$. 
\par
Moreover, the shape operator of $A+B$ can be expressed by the following formula:
\begin{equation}\label{KraPar}
S_{A+B}(u)=\left[I+S_A(u)^{-1}S_B(u)\right]^{-1}S_B(u).
\end{equation}
\end{thm}
The proof of this theorem can be repeated step by step also in the case in which the $C^\infty$-regularity of $A$ is replaced by its $C^2$-regularity: one then gets that the boundary of $A+B$ is
of class $C^2$ and that \eqref{KraPar} holds, as well.
\par 
Thus, our aim is now to show that $K$-dense bodies are strongly convex; then, by Theorem \ref{KrPa},  we will gain 
the necessary regularity of $K$ that gives a meaning to \eqref{KraPar} with $A=G$ and $B=-G$.  
\par In order to do this, for $x\in\pa G$ we shall study the asymptotic behavior of $V(G\setminus(x+rK))$ as $r\to 1^-$. As we shall see, if we want to express 
$V(G\setminus(x+rK))$ in terms of the shape operator of $\pa G$ at some point $\ovr x\in\pa G$, it is important to make sure that $G$ shares with the boundary of $x+K$ only one point.
We observe that this is not always the case: indeed, consider the Releaux triangle as the set $G$ and let $x$ denote one of its vertices; then, $K=G-G$ is a ball and $G\cap (x+K)$ is one of the arcs constituting the triangle's boundary; hence, so to speak, $G\setminus(x+rK)$ can not be localized around any point of $\pa G$.
\par
Notice that such a $G$ is strictly convex, but $\pa G$ is not differentiable. Likewise, if we consider differentiable bodies which are not strictly convex, we can still provide an example of the same phenomenon: in fact, it is enough to set $G=B+Q$, where $B$ is the unit ball and $Q$ is the unit square. 
\par
The following lemma shows that we can get the desired result, if we assume that $G$ is  both differentiable and strictly convex.

\begin{lem}\label{lemma1}
Let $G$ be a strictly convex body with differentiable boundary and set $K=G-G$, then for each $x\in \pa G$ the set $\pa (x+K)\cap G$ consists of only one point $\ovr x\in\pa G$ characterized by $\nu_K(\ovr x-x)=-\nu_G(x)$.

\end{lem}
\begin{proof}
Let $z\in\partial K\cap (G-x)$ and let $u=\nu_K(z)$. Clearly $z+x\in\pa G$ and, since the $G-x$ is contained in $K$ and touches $K$ at $z$ from inside, then $\nu_G (z+x)=u$.
Since $K=G-G$, we have 
\[
\begin{aligned}
h_G(u)+h_G(-u)&=h_K(u)=\langle z,u\rangle=
\langle z+x,u\rangle+\langle x,-u\rangle\\
&=h_G(u)+\langle x,-u\rangle.
\end{aligned}
\]
Thus, $h_G(-u)=\langle x,-u\rangle$, that is $\nu_G(x)=-u$. It is then enough to set $\ovr x=z+x$.
\par
Now, suppose that there exists another point $z'$ such that $z'\in\partial K\cap (G-x)$ and set $u'=\nu_K(z')$; by the same argument, we get that $\nu_G(x)=-u'$, and hence $u=u'$. Since $K$ is strictly convex (being $G$ so), we finally find $z=z'$.
\end{proof}

The following lemma is helpful to prove that a $K$-dense set is positively curved.

\begin{lem}\label{prop1}
Let $G$ be a strictly convex body with boundary of class $C^2$ and let $K=G-G$.  For $x\in\pa G$ and $\ovr x\in\pa G$ such that $u=\nu_G(x)=-\nu_G(\ovr x)$,
It holds:
\begin{itemize}
\item[(i)]
if $\ka_G(u)=0$, then
\[
\liminf_{r\to 1^-}\,\frac{V(G\setminus ( x+r K)}{(1-r)^{\frac{N+1}{2}}}
=+\infty;
\]
\item[(ii)] if $\ka_G(u)>0$, then
\[
\limsup_{r\to 1^-}\,\frac{V(G\setminus (x+r K)}{(1-r)^{\frac{N+1}{2}}}
\le \frac{2\om_{N-1}}{N^2-1}\,\ka_G(u)\,h_K(u)^{\frac{N+1}2}(1+\La)^{\frac{N-1}2},
\]
where $\La$ is the maximal principal curvature of $\pa G$ at $x$. 
\end{itemize}
\end{lem}
\begin{proof}
First, notice that, by the above lemma, our choice of $x$ and $\ovr x$ ensures that $\{\ovr x\}=\pa(x+K)\cap G$.
Without loss of generality, we can always assume that $\ovr x=0$ and that $u=(0, 0,\dots, -1)$;
then, in a neighborhood of $\ovr x$, $\pa G$ can be parametrized by
\begin{equation}
\label{localG}
y_N=\langle S_G(u)\,y, y\rangle+o(|y|^2) \ \mbox{ as } \ |y|\to 0,
\end{equation}
where $y=(y_1,\dots, y_{N-1})$ ranges in the tangent space to $\pa G$ at $\ovr x$.
\par
(i) Set $\ep=1-r$. Let $\ep_n$ be an infinitesimal sequence of positive numbers on which the limit in (i) is attained and, to simplify notations, set
$G_n=G\setminus ( x+(1-\ep_n)K)$;
\eqref{localG} suggests that, by possibly extracting a subsequence from $\ep_n$,
we can fit in $G_n$
the set $E_n$ bounded by the paraboloid
$$
y_N=\langle S_G(u)\,y, y\rangle +\frac1n\,|y|^2
$$
and the hyperplane $\ovr x+\ep_n\,h_K(u)\,u+u^\perp$ supporting the set 
$x+(1-\ep_n)K$ at the point whose outer unit normal coincides with $u$. 
In our coordinates,
$$
E_n=\{(y,y_N): \langle S_G(u)\,y, y\rangle +\frac1n\,|y|^2<y_N<\ep_n\,h_K(u) \}
$$
and $E_n\subseteq G_n$.
\par
Thus, by Fubini's theorem and some calculations, we get:
\begin{eqnarray*}
V(G_n)\ge V(E_n)=
\int_0^{\ep_n\,h_K(u)} \mathcal H^{N-1}\left(\left\{y:\langle [S_G(u)+1/n\,I]\,y, y\rangle \leq t\right\}\right)\,dt=\\
\frac{\om_{N-1}}{(N-1) \det\left[S_G(u)+1/n \,I\right]^{\frac12}}\int_0^{\ep_n\,h_K(u)} t^{\frac{N-1}{2}}dt=\frac{2\om_{N-1}\,\ep_n^{\frac{N+1}{2}}\,h_K(u)^{\frac{N+1}{2}}}{(N^2-1) \det\left[ S_G(u)+1/n\, I\right]^{\frac12}}.
\end{eqnarray*}
\par
Therefore,
\begin{eqnarray*}
&&\liminf_{\ep\to 0^+}\frac{V(G\setminus ( x+(1-\ep)K)}{\ep^{\frac{N+1}{2}}}=
\lim_{n\to\infty}{\ep_n^{-\frac{N+1}{2}}}{V(G_n)}\ge \\
&&\lim_{n\to\infty}\frac{2\,\om_{N-1}\,h_K(u)^{\frac{N+1}{2}}}{(N+1)\sqrt{\det\left( S_G(u)+1/n\, I\right)}}
=+\infty,
\end{eqnarray*}
since $\det S_G(u)=\ka_G(u)=0$.
\par
(ii) We shall obtain the desired inequality by observing that the domain $G\setminus (x+(1-\ep)K)$ can be contained in the region $F_{\ep,\de}$ bounded by two paraboloids: one outside $G$ and tangent to $\pa G$ at $\ovr x$, the other one tangent to the boundary of $x+(1-\ep)K$ from inside.
In order to show it, we assume as before that $\ovr x=0$ and $u=-e_N$ and, moreover, that
$S_G(u)=I$ (this can be done since the affine tranformation $S_G(u)$ is invertible, being $\det S_G(u)=\ka_G(u)>0$): the desired formula will then be obtained by multiplying 
the right-hand side of \eqref{formula-det=1} by the factor $\ka_G(u)$.
\par
We proceed to contruct $F_{\ep,\de}$. We choose any number $\la>0$ such that $\la\, I>S_G(-u)$, that is such that $\la>\La$.
Since $\ka_G(u)>0$, Theorem \ref{KrPa} imply that $\pa K$ is twice differentiable at $u$; moreover equation \eqref{KraPar} turns into
\[
S_K(u)<\frac{\lambda}{1+\lambda}\,I;
\]
hence,
\[
S_{(1-\ep)K}(u)<\frac{\lambda}{(1+\lambda)(1-\ep)}\,I.
\]
For $\ep>0$ sufficiently small, we define $F_{\ep, \de}$ as
\[
F_{\ep, \de}=\left\{(y, y_N): \de\,|y|^2\leq y_N\leq \ep\,h_K(u)+\frac{\la}{(1+\la)(1-\ep)}\,|y-\ep x_*|^2\right\},
\]
where $\de$ is chosen in the interval $(\frac{\la}{(1+\la)(1-\ep)},1)$ and $x_*$ is the 
projection of $x$ on the tangent space to $\pa G$ at $\ovr x$; in this way, 
\[
G\setminus (x+(1-\ep)K)\subset F_{\ep, \de}.
\]
Indeed, equation \eqref{localG} guarantees that the above inclusion holds, at least inside a small neighborhood of $\ovr x$; however, by Lemma \ref{lemma1}, we know that $G\setminus (x+(1-\ep)K)$ is contained in a ball $B_r$ around $z$ whose radius $r=r(\ep)$ tends to $0$ as $\ep\to 0$.
\par
By using the rescaling $(y, y_N)=(\sqrt{\ep}\,\xi, \ep\,\xi_N)$, we obtain that
$
V(F_{\ep,\de})=\ep^{\frac{N+1}2}\,V(F'_{\ep,\de}),
$
where 
$$
F'_{\ep, \de}=\left\{(\xi, \xi_N): \de\,|\xi|^2\leq \xi_N\leq h_K(u)+\frac{\la}{(1+\la)(1-\ep)}\,|\xi-\sqrt{\ep}\,x_*|^2\right\},
$$
and it is easy to show that
$
V(F'_{\ep,\de})\to V(F'_{0,\de}).
$
By a straightforward computation of $V(F'_{0,\de})$, we get that
\begin{equation*}
\limsup_{\ep\to 0}\frac{V(G\setminus (x+(1-\ep)K))}{\ep^{\frac{N+1}{2}}}\leq 
\frac{2\om_{N-1}}{N^2-1}\,\frac{h_K(u)^{\frac{N+1}2}}{(\de-\frac{\la}{1+\la})^{\frac{N-1}2}},
\end{equation*}
and minimizing the right-hand side of this formula for $\la/(1+\la)<\de<1$ and $\la>\La$  then gives:
\begin{equation}
\label{formula-det=1}
\limsup_{\ep\to 0}\frac{V(G\setminus (x+(1-\ep)K))}{\ep^{\frac{N+1}{2}}}\leq 
\frac{2\om_{N-1}}{N^2-1}\,h_K(u)^{\frac{N+1}2}(1+\La)^{\frac{N-1}2}.
\end{equation}

\end{proof}

\begin{cor}\label{regularity}
If $G$ is $K$-dense, then $\pa K$ is of class $C^2$ and every point of $\pa K$ is a point of strong convexity. Moreover, 
\begin{equation}\label{KP1}
S_K(u)=\left[I+S_G(u)^{-1}S_G(-u)\right]^{-1}S_G(-u).
\end{equation} 
\end{cor}

\begin{proof}
Since $G$ is $K$-dense, then the limits in items (i) and (ii) in Lemma \ref{prop1}
do not depend on the
particular point $x\in\pa G$; in other words, they must be constant functions on $\pa G$.
Since $G$ is a convex body and $\pa G$ is of class $C^2$, then $\ka_G$ is not identically zero; hence, the limit in item (ii) of Lemma \ref{prop1} is a finite constant. As a consequence, item (i) of the same lemma implies that $\ka_G>0$ (and hence $S_G>0$) on $\pa G$. Formula \eqref{KP1} is then  a straightforward consequence of Theorem \ref{KrPa}.
\end{proof}

\begin{thm}\label{thm}
Let $G$ be a strongly convex body with boundary of class $C^2$ and  set $K=G-G$. Chose $x$, $u$ and $\ovr x$ as in Lemma \ref{prop1}; then
\[
\lim_{r\to 1^-}\frac{V(G\setminus ( x+r K))}{(1-r)^{\frac{N+1}{2}}}=\frac{2\omega_{N-1}\, h_K(u)^{\frac{N+1}{2}}}{(N^2-1) \det[S_G(u)-S_K(u)]^{\frac 1 2}}.
\]
\end{thm}
\begin{proof}
Again we set $\ep=1-r$. We begin by showing that
\begin{equation}
\label{min}
\limsup_{\varepsilon\rightarrow 0^+}\frac{V(G\setminus ( x+(1-\ep) K))}{\ep^{\frac{N+1}{2}}}\leq 
\frac{2\omega_{N-1}\, h_K(u)^{\frac{N+1}{2}}}{(N^2-1) \det[S_G(u)-S_K(u)]^{\frac 1 2}}.
\end{equation}
As in the proof of Lemma \ref{prop1}, without loss of generality,  we can set $u=-e_N$ and $\ovr x=0$. 
\par
It is clear that $\ep x\in\pa(x+(1-\ep)K)$ and that $u$ is the unit normal  to $\partial(x+(1-\ep) K)$ at that point; also, by a scaling argument, we know that
\[
S_{x+(1-\ep)K}(u)=\frac{S_K(u)}{1-\varepsilon}.
\]
Notice that formula \eqref{KP1} implies that $S_G(u)>S_K(u)$; hence, we can chose $\overline n\in\mathbb N$ such that
\begin{equation}\label{cond}
\frac{S_G(u)-S_K(u)}{4}>\frac{I}{\overline n}.
\end{equation}
\par
In order to get an estimate from above for $V(G\setminus (x+(1-\ep) K))$
we construct a set $C_{\varepsilon,n}$ containing $G\setminus (x+(1-\ep) K)$.
In fact, for $n>\ovr n$ we set 
\begin{multline}\nonumber
C_{\ep,n}=\big \{(y,y_N): \langle\left (S_G(u)- n^{-1} I\right)y, y\rangle<y_N<\\
\ep h_K(u)+\langle\left[(1-\ep)^{-1}S_K(u)+n^{-1} I\right](y-x_*),(y-x_*)\rangle\big\},
\end{multline}
where $x_*$ denotes the projection of $x$ on $u^\perp$;
$C_{\ep,n}$ is the region bounded by two paraboloids, one touching $\pa G$ at $\ovr x$ from below, the other one touching $\pa (x+(1-\ep)K)$ at $\ep x$ from above and,
for $\ep$ small enough, we have:
\[
G\setminus (x+(1-\varepsilon) K)\subset C_{\varepsilon,n}.
\]
Also, condition \eqref{cond} guarantees that
\[
S_G(u)-\frac{I}{n}>\frac{S_K(u)}{1-\varepsilon}+\frac{I}{n}>0,
\]
thus forcing $C_{\varepsilon,n}$ to be bounded.
\par
The usual change of variables $(y, y_N)=(\sqrt{\ep}\,\xi, \ep\,\xi_N)$ gives that $V(C_{\ep,n})=\ep^{\frac{N+1}2}\,V(C'_{\ep,n})$, where
\begin{multline}\nonumber
C'_{\varepsilon,n}=\big \{(\xi,\xi_N): \langle\left[S_G(u)- n^{-1} I\right]\xi, \xi\rangle<\xi_N<
h_K(u)+\\
\langle\left[(1-\varepsilon)^{-1}S_K(u)+n^{-1} I\right](\xi-\sqrt\ep x_*),(\xi-\sqrt\ep x_*)\rangle\big\},
\end{multline}
\par
Since clearly $V(C'_{\ep, n})\to V(C'_{0,n})$ as $\ep\to 0$, a straightforward computation gives:
\begin{multline}
\label{upperbound}
\limsup_{\varepsilon\rightarrow 0^+}\frac{V(G\setminus ( \overline z+(1-\varepsilon) K))}{\varepsilon^{\frac{N+1}{2}}}\leq V(C'_{0,n})=\\
\frac{2\omega_{N-1}\, h_K(u)^{\frac{N+1}{2}}}{(N^2-1)\,\mathrm{det}\left[S_G(u)-S_K(u)-2/n\,I\right]^{\frac{1}{2}}}.
\end{multline}
\par
Since the right-hand side in \eqref{upperbound} is independent on $n$, \eqref{min} follows at once by taking the limit for $n\to\infty$.
\par
The converse inequality,
\[
\liminf_{\varepsilon\rightarrow 0^+}\frac{V(G\setminus ( \overline z+(1-\varepsilon) K))}{\varepsilon^{\frac{N+1}{2}}}\geq\frac{2\omega_{N-1}\, h_K(u)^{\frac{N+1}{2}}}{(N^2-1)\,\mathrm{det}\left[S_G(u)-S_K(u)-2/n\,I\right]^{\frac{1}{2}}},
\]
is proved by using the same strategy used for \eqref{min}: we choose $\overline n$ such that
\[
S_G(u)>\frac{I}{\overline n}
\]
and then we construct, for $n>\overline n$ and $\varepsilon$ small, a set $D_{\ep,n}\subseteq G\setminus (x+(1-\ep) K)$:
\begin{multline}\nonumber
D_{\varepsilon,n}=\big \{(y,y_N): \langle\left (S_G(u)+ n^{-1} I\right) y, y\rangle<y_N<\\
< \varepsilon h_K(u)+\langle\left[(1-\varepsilon)^{-1}S_K(u)-n^{-1} I\right](y-\ep x_*),(y-\ep x_*)\rangle\big\}.
\end{multline}
As before, the usual rescaling gives 
\[
\liminf_{\varepsilon\rightarrow 0^+}\frac{V(G\setminus (x+(1-\ep)K))}{\ep^{\frac{N+1}{2}}}\geq
\frac{2\omega_{N-1}\,h_K(u)^{\frac{N+1}{2}}}{(N^2-1)\,\mathrm{det}\left[S_G(u)-S_K(u)+2/n\,I\right]^{\frac{1}{2}}}.
\]
Again, we conclude by taking the limit for $n\to\infty$.
\end{proof}

\begin{cor}\label{cor}
Let $G$ be a $K$-dense body, then \eqref{r-large} holds with the coefficient $W(x)$ given by \eqref{W1}. 
In particular, the function defined by
\begin{equation}\label{coreq}
\frac{h_K(u)^{\frac{N+1}{2}}}{\det[S_G(u)-S_K(u)]^{\frac 1 2}}\,, \ u\in\SSN,
\end{equation}
is constant on $\SSN$.
\end{cor}
\begin{proof}
Corollary \ref{regularity} ensures that a $K$-dense body satisfies the assumptions of Theorem \ref{thm}. Since $V(G\cap (x+r\,K))=V(G)-V(G\setminus(x+r\,K))$, clearly
\[
W(x)=\lim_{r\to1^-}\frac{V(G\setminus(x+r\,K))}{(1-r)^{\frac{N+1}{2}}}
\]
and $W(x)$ must be constant for $x\in\pa G$. Since $G$ is strictly convex, \eqref{coreq} then follows  
from the suriectivity of the Gauss map.
\end{proof}

Now, we are going to show that if $G$ is $K$-dense, then $G$ and $K$ must be equal
up to homotheties.
\begin{pro}
Let $G$ be a $K$-dense body, then $\ka_G(u)=\ka_G(-u)$.
\end{pro}
\begin{proof}
Let $u\in\mathbb S^{N-1}$ and $L=L_u$ be a linear map of $\RN$ in itself, which leaves unchanged the unit vector $u$ and whose restriction to $u^\perp$ equals  $S_G(u)^{-\frac 1 2}$.
\par
First, notice that, as an easy consequence of \eqref{Kdense}, the set $LG$ is $LK$-dense, so that Corollary \ref{cor} holds for this set; in particular, \eqref{coreq} implies:
\begin{multline}\label{sym}
h_{LK}(-u)^{\frac{N+1}{2}}\{\det[S_{LG}(-u)-S_{LK}(-u)]\}^{-\frac 1 2}= \\
h_{LK}(u)^{\frac{N+1}{2}}\{\det[S_{LG}(u)-S_{LK}(u)]\}^{-\frac 1 2}.
\end{multline}
Secondly, we know that $K$ is centrally symmetric, and so must be $LK$; then, $S_{LK}(u)=S_{LK}(-u)$ and $h_{LK}(u)=h_{LK}(-u)$; \eqref{sym} then becomes
\begin{equation}\label{sym2}
\det[S_{LG}(-u)-S_{LK}(u)]=\det[S_{LG}(u)-S_{LK}(u)].
\end{equation}
As we shall see, this condition together with equation \eqref{KraPar} is enough to prove that
\[
\det[S_{LG}(u)]=\det[S_{LG}(-u)].
\]
Indeed, by plugging \eqref{KraPar} into \eqref{sym2} we get
\begin{multline}\label{sym3}
\det\left(S_{LG}(-u)-\left[I+S_{LG}(u)^{-1}S_{LG}(-u)\right]^{-1}S_{LG}(-u)\right)=
\\
\det\left(S_{LG}(u)-\left[I+S_{LG}(u)^{-1}S_{LG}(-u)\right]^{-1}S_{LG}(-u)\right);
\end{multline}
furthermore, our chioice of the affine transformation $L$  ensures that
\[
S_{LG(u)}=I,
\]
and
\begin{equation}\label{HLG}
S_{LG}(-u)=S_G(u)^{-\frac 1 2}S_G(-u)S_G(u)^{-\frac 1 2}.
\end{equation}
Equation \eqref{sym3} then turns into
\begin{multline}\label{sym4}
\det\left(S_{LG}(-u)-\left[ I+S_{LG}(-u)\right]^{-1}S_{LG}(-u)\right)=
\\
\det\left(I-\left[I+S_{LG}(-u)\right]^{-1}S_{LG}(-u)\right);
\end{multline}
by multiplying both sides of \eqref{sym4} by $\det[I+S_{LG}(-u)]$ and using Binet's identity, we get
\[
\det[S_{LG}(-u)^2]=1.
\]
\par
Finally, \eqref{HLG} gives that $\det[S_G(u)]=\det[S_G(-u)]$, that is $\ka_G(u)=\ka_G(-u)$.
\end{proof}

\begin{cor}
Let $G$ be $K$-dense, then $G$ is symmetric and $K=2G$.
\end{cor}
\begin{proof}
The two bodies $G-G$ and $2G$ have the same Gaussian curvature as a function on $\mathbb S^{N-1}$; thus, they only differ by a translation.
\end{proof}
\par
The following theorem and Petty's characterization of ellipsoids \cite{Pe} complete
the proof of Theorem \ref{th:main}.

\begin{thm}
Let $G$ be a $K$-dense set. Then, for every $x\in\pa G$ it holds that
\[
\lim_{r\rightarrow 1^-}\frac{V(G\setminus (x+r K))}{(1-r)^{\frac{N+1}{2}}}=\frac{2\sqrt{2}\,\om_{N-1}
\,h_K(u)^{\frac{N+1}{2}}}{(N+1)\,\det[S_G(u)]^{\frac12}} \ \mbox{ with } \ u=\nu(\ovr x)
\]
and $\{\ovr x\}=\pa G\cap (x+K)$.
\par
In particular, there exists a positive constant $c$, depending only on $N$, such that
\[
\ka_G(u)=c\,h_G(u)^{N+1} \ \mbox{ for every } \ u\in\SSN.
\]
\par
Therefore, $G$ must be an ellipsoid.
\end{thm}

\end{document}